\documentclass[12pt]{amsart}
\usepackage{amsmath,amssymb,amsthm,mathtools}
\usepackage{mathrsfs,bm}
\usepackage{thmtools,thm-restate}
\usepackage{tikz-cd}
\usepackage{tikz}
\usepackage{enumitem}
\usepackage{graphicx}
\usetikzlibrary{matrix,arrows.meta,bending}
\usepackage{color}
\usepackage[numbers]{natbib}
\usepackage{booktabs}
\usepackage{longtable}
\usepackage{pgfplots}
\pgfplotsset{compat=1.14} 
\usetikzlibrary{arrows.meta,decorations.pathreplacing,positioning,calc} 

\usepackage[a4paper, left=3cm, right=3cm, top=3cm, bottom=3cm, footskip=15mm]{geometry}

\newtheorem{theorem}{Theorem}[section]

\newtheorem{proposition}[theorem]{Proposition}
\newtheorem{corollary}[theorem]{Corollary}
\newtheorem{conjecture}[theorem]{Conjecture}
\theoremstyle{definition}
\newtheorem{definition}[theorem]{Definition}

\theoremstyle{remark}
\newtheorem{remark}[theorem]{Remark}

\numberwithin{equation}{section}

\newcommand{\inn}{~ \hat{\in}~ }

\makeindex

\usepackage{fancyhdr}
\usepackage{calc} 

\AtBeginDocument{%
  \setlength{\textwidth}{\dimexpr\paperwidth - 6cm\relax}%
  \setlength{\oddsidemargin}{\dimexpr 3cm - 1in\relax}%
  \setlength{\evensidemargin}{\dimexpr 3cm - 1in\relax}%
  \setlength{\marginparwidth}{2.5cm}%
}

\begin{document}

\title{On the Erd\H{o}s-Tur\'{a}n additive base conjecture}

\author{T. Agama}
\address{Department of Mathematics, African Institute for Mathematical science, Ghana
}
\email{theophilus@aims.edu.gh/emperordagama@yahoo.com}



\date{\today}



\begin{abstract}
In this paper, we formulate and prove several variants of the Erd\H{o}s-Tur\'{a}n additive bases conjecture.
\end{abstract}

\maketitle

\section{Introduction}

The study of additive bases, sets $\mathbb{B}\subset\mathbb{N}$ for which every sufficiently large integer can be represented as a sum of a fixed number of elements of $\mathbb{B}$, lies at the crossroads of classical additive number theory and modern additive combinatorics. Fundamental motivating examples include the (weak) Goldbach problem for primes (the ternary Goldbach theorem of Helfgott), which illustrates how deep analytic and combinatorial methods interact with structural properties of a sparse set. In particular, the resolution of the ternary Goldbach conjecture by Helfgott provides a canonical illustration of the kind of global additive covering phenomena that we seek to understand. \cite{helfgott2013ternary}\\

Among the central open problems in this area is the classical conjecture of Erd\H{o}s and Tur\'an, commonly referred to as the Erd\H{o}s--Tur\'an additive bases conjecture: if a set $\mathbb{B}\subset\mathbb{N}$ satisfies
$$
r_{\mathbb{B}}(n):=\#\{(a,b)\in\mathbb{B}^2:a+b=n\}>0
$$
for all sufficiently large $n$, then $r_{\mathbb{B}}(n)$ is unbounded as $n\to\infty$. Despite many partial results and a substantial body of surrounding theory developed in the second half of the twentieth century and beyond (for background and techniques, see, e.g., \cite{erdos1941problem,tao2006additive}), a full resolution remains elusive and, in particular, resists approaches that treat bases only through coarse density measures.\\

In this paper, we introduce and exploit a new geometric--combinatorial framework, the \emph{circle of partition} (CoP), which refines the usual perspective on additive representations by organizing representations $a+b=n$ as \emph{axes} (ordered pairs of complementary weights) on a combinatorial circle indexed by $n$. Three central quantitative notions naturally emerge from this viewpoint and form the technical core of our approach: (i) the density of points on CoPs relative to a given subset $\mathbb{M}\subset\mathbb{N}$, which connects classical asymptotic density with the proportion of axes that meet $\mathbb{M}$; (ii) the classification of the behaviour of CoP into \emph{ascending}, \emph{descending} and \emph{stationary} regimes as the generator $n$ varies, a coarse dynamical invariant that encodes how the supply of representations changes with scale; and (iii) the \emph{$\ell$-th fold energy} of a collection of CoPs, an aggregate statistic (summed over power-sequences $n^\ell$) that plays a role analogous to energy functionals in additive combinatorics and which is inspired by the combinatorial energy methods of Ruzsa and others. The combination of these tools allows us to reframe the Erd\H{o}s--Tur\'an conjecture in a way that makes certain density- and energy-based obstructions explicit, and to prove several natural variants of the conjecture under additional quantitative hypotheses.\\

Technically, our main contributions are as follows.\\
\begin{itemize}
  \item We formalize the CoP construction and establish its basic combinatorial geometry: points, axes, centers, and chords, together with elementary uniqueness and counting lemmas for axes (Section \ref{sec:cop}). These foundational properties show that the mapping between additive representations and axes of CoPs is bijective and stable under natural operations, which allows coarse asymptotic counting to be translated into structural statements about CoPs.
  \bigskip
  
  \item We introduce and compare two density notions: the classical asymptotic density of a subset $\mathbb{H}\subset\mathbb{N}$ and the density of axes in CoPs that meet $\mathbb{H}$. Proposition \ref{inequality} gives sharp inequalities linking these two densities and hence precisely quantifies how ordinary density bounds control the proportion of axes captured by a set. This comparison is the springboard for several corollaries showing that ``very dense'' sets necessarily generate many representations. These density comparisons may be viewed as discrete analogs of the density-versus-energy tradeoffs familiar from modern additive combinatorics \cite{tao2006additive,nathanson1996additive}.
  \bigskip
  
  \item We define ascending/descending/stationary CoPs and show (Theorem \ref{ascendingspots}) that under mild asymmetry hypotheses between a set and its complement, ascending behaviour must occur at infinitely many scales. Intuitively, if a set occupies a positive proportion of integers and its complement is not too large relative to it, then representations inside the set cannot remain uniformly scarce at all large scales.
  \bigskip
  
  \item To reach beyond the positive-density hypotheses, we introduce the $\ell$-th fold energy $\mathcal{E}(\ell,\mathbb{M})$ and relate the divergence of this energy to the existence of infinitely many ascending spots (Proposition \ref{spotenergy}). The energy viewpoint captures cumulative structural richness across scales and is particularly effective for sequences that are thin but arithmetically structured (compare the role of energies in sumset estimates and in Pl\"unnecke--Ruzsa theory \cite{ruzsa1994plunnecke}).
  \bigskip
  
  \item Combining these tools, we obtain several variants of the Erd\H{o}s--Tur\'an statement: density-constrained versions (very dense sequences are shown to produce unbounded representation counts) and energy-driven versions (divergent $\ell$-fold energy implies infinitely many large representation counts).
\end{itemize}
\bigskip

\textbf{Relation to prior work.} Our perspective draws on and synthesizes ideas from classical additive number theory and modern additive combinatorics. The comparison inequalities between ordinary density and CoP-density are inspired by elementary counting arguments that appear in the treatment of Nathanson of representations by sumsets \cite{nathanson1996additive}; the energy viewpoint and the emphasis on cumulative scale-summed statistics are in the spirit of the Ruzsa energy methods and Pl\"unnecke-type inequalities \cite{ruzsa1994plunnecke}. At a philosophical level, the CoP formalism provides a geometric bookkeeping device that complements Fourier-analytic and combinatorial tools used in recent advances and naturally sits alongside the frameworks developed in the literature on additive bases and covering systems (see, e.g., \cite{tao2006additive,nathanson1996additive}).
\bigskip

\textbf{Organization of the paper.} In Section \ref{sec:cop}, we define the Circle of Partition and record basic combinatorial properties, uniqueness of axes, and elementary counting results. Section \ref{sec:density} introduces the CoP-density and establishes the comparison inequalities with the classical density; Section \ref{sec:ascdesc} is devoted to the ascending /descending /stationary classification and contains Theorem \ref{ascendingspots} together with corollaries that give density-based variants of the Erd\H{o}s--Tur\'an additive base conjecture. Section \ref{sec:energy} defines the $\ell$-th fold energy, analyzes its basic behaviour (finite and infinite regimes) and proves Proposition \ref{spotenergy}; Section \ref{sec:main} collects the main theorems that realize variants of the Erd\H{o}s--Tur\'an conjecture under density or energy hypotheses and discusses examples, limits, and possible refinements.
\bigskip

\noindent\textbf{Notation and conventions.} Throughout, we write $\mathbb{N}=\{1,2,\dots\}$, $\mathbb{N}_n=\{1,\dots,n\}$, and for a set $\mathbb{S}\subset\mathbb{N}$, we denote by $r_{\mathbb{S}}(n)$ the number of ordered representations of $n$ as the sum of two elements of $\mathbb{S}$. As usual $\mathcal{D}(\mathbb{S})$ denotes (when it exists) the asymptotic density of $\mathbb{S}$. The unstated logarithmic constants are absolute. Complete proofs of the statements summarized above appear in the indicated sections.

\section{Problem Statement}

Let $\mathbb{S}$ be a subset of the natural numbers $\mathbb{N}$ and $k\in \mathbb{N}$ be fixed. Then $\mathbb{S}$ is said to be an \textbf{additive base} of order $k$ if every natural number can be expressed as a sum of $k$ elements of $\mathbb{S}$. The weak Goldbach conjecture suggests that the set of prime numbers is an additive base of \emph{order} three \cite{helfgott2013ternary}. The Erd\H{o}s-Tur\'{a}n additive base conjecture is the assertion that all additive bases qualify very sufficiently to be an additive additive base in an asymptotic sense. In particular, we have the following conjecture of Erd\H{o}s and Tur\'{a}n (see \cite{erdos1941problem})

\begin{conjecture}[Erd\H{o}s-Tur\'{a}n]
Let $\mathbb{B}\subset \mathbb{N}$ and consider
\begin{align}
  r_{\mathbb{B}}(n):=\# \left \{(a,b)\in \mathbb{B}^2~|~ a+b=n\right\}.\nonumber
\end{align}
If $r_{\mathbb{B}}(n)>0$ for all sufficiently large values of $n$, then 
\begin{align}
    \limsup \limits_{n\longrightarrow \infty}~r_{\mathbb{B}}(n)=\infty.\nonumber
\end{align}
\end{conjecture}

This conjecture has garnered the attention of many authors, but remains unresolved \cite{tao2006additive}. By introducing the language of \emph{Circles of Partition} and associated statistics, we reformulate the conjecture in the following way 

\begin{conjecture}[Erd\H{o}s-Tur\'{a}n]
Let $\mathbb{B}\subset \mathbb{N}$ and consider 
\begin{align}
   \mathcal{G}_{\mathbb{B}}(n)=\nu(n,\mathbb{B})=\# \left \{ \mathbb{L}_{[x],[y]}\inn \mathcal{C}(n,\mathbb{B})\right \}.\nonumber
\end{align}
If $\mathcal{G}_{\mathbb{B}}(n)>0$ for all sufficiently large values of $n$, then
\begin{align}
    \limsup \limits_{n\longrightarrow \infty}\mathcal{G}_{\mathbb{B}}(n)=\infty.\nonumber
\end{align}
\end{conjecture}
Using the notion of the density of circles of partition, the notion of ascending, descending, and stationary circles of partition, and the $l$ th fold energy of the circle of partitions, we study the Erd\H{o}s-Tur\'{a}n additive bases conjecture.

\section{The Circle of Partition}\label{sec:cop}

In this section, we introduce the notion of the circle of partition.

\begin{definition}\label{major}
Let $n\in \mathbb{N}$ and $\mathbb{M}\subset \mathbb{N}$. We denote the \emph{Circle of Partition} generated by $n$ with respect to the subset $\mathbb{M}$ by 
\begin{align}
\mathcal{C}(n,\mathbb{M})=\left\{[x]~:~ x,y\in \mathbb{M},n=x+y\right\}.\nonumber
\end{align}
In the following section, will abbreviate this as CoP. We call members of $\mathcal{C}(n,\mathbb{M})$ \emph{points} and denote them by $[x]$. For the special case $\mathbb{M}=\mathbb{N}$, we denote CoP in the short form $\mathcal{C}(n)$. 
\end{definition} 

\begin{definition}\label{axis}
We denote the line that joins the point $[x]$ and $[y]$ by $\mathbb{L}_{[x],[y]}$ and call it an \emph{axis} of the CoP $\mathcal{C}(n,\mathbb{M})$ if and only if $x+y=n$. We say that the axis point $[y]$ is an \emph{axis partner} of the axis point $[x]$ and vice versa. We do not distinguish between the axes $\mathbb{L}_{[x],[y]}$ and $\mathbb{L}_{[y],[x]}$, because it is essentially the same axis. The point $[x]\in \mathcal{C}(n,\mathbb{M})$ such that $2x=n$ is the \emph{center} of the CoP.
If it exists, then it is their only point which is not an axis point. The line joining any two arbitrary points that are not axes partners on the CoP will be called a \emph{chord} of the CoP. The length of the chord joining the points $[x],[y]\in \mathcal{C}(n,\mathbb{M})$, denoted by $\mathcal{D}([x],[y])$, is the measure
\begin{align}
\mathcal{D}([x],[y])=|x-y|.\nonumber
\end{align}
\end{definition}
\bigskip

It is important to emphasize that the \emph{median} of the weights of each co-axis point coincides with the center of the underlying CoP if it exists. That is, given all the axes of CoP $\mathcal{C}(n,\mathbb{M})$
\begin{align}
\mathbb{L}_{[u_1],[v_1]},\mathbb{L}_{[u_2],[v_2]},\ldots, \mathbb{L}_{[u_k],[v_k]}\nonumber
\end{align}
we deduce the following relations
\begin{align}
\frac{u_1+v_1}{2}=\frac{u_2+v_2}{2}=\cdots=\frac{u_k+v_k}{2}=\frac{n}{2}\nonumber
\end{align}
which is equivalent to the following conditions: for any of the pair of axes $\mathbb{L}_{[u_i],[v_i]},\mathbb{L}_{[u_j],[v_j]}$ for $1\leq i,j\leq k$
\begin{align}
\mathcal{D}([u_i],[u_j])=\mathcal{D}([v_i],[v_j])\nonumber
\end{align}
and 
\begin{align}
\mathcal{D}([v_j],[u_i])=\mathcal{D}([u_j],[v_i]).\nonumber
\end{align}
\bigskip

The above framework could be seen as a criterion to determine the plausibility of partitioning in a specified set. Indeed, this feasibility is trivial when we take the set $\mathbb{M}$ to be the set of positive integers $\mathbb{N}$. The situation becomes less trivial and interesting when we take the set $\mathbb{M}$ to be a special subset of natural numbers $\mathbb{N}$. In this case, the corresponding CoP $\mathcal{C}(n,\mathbb{M})$ may not always be non-empty for all $n\in \mathbb{N}$. One archetype of problems of this flavour is the binary Goldbach conjecture, when we take the base set $\mathbb{M}$ to be the set of all prime numbers $\mathbb{P}$. One could imagine the same level of difficulty when we extend our base set to other special subsets of the natural numbers.

\begin{remark}
We observe that a typical CoP may not always have a center. In the case of an absence of a center, we say that the circle has a deleted center. It is easy to observe that the CoP $\mathcal{C}(n)$ contains all points whose weights are positive integers from $1$ to $n-1$ inclusive: 
$$
\mathcal{C}(n)=\left \{[x]\mid~x\in \mathbb{N},x<n\right \}.
$$ 
Therefore, CoP $\mathcal{C}(n)$ has $\left \lfloor \frac{n-1}{2}\right \rfloor$ different axes.
\end{remark}
\bigskip

In the sequel, we will denote the assignment of an axis $\mathbb{L}_{[x],[y]}$ to a CoP $\mathcal{C}(n,\mathbb{M})$ by
$$
\mathbb{L}_{[x],[y]}\inn \mathcal{C}(n,\mathbb{M})
$$
which means
$$
[x],[y]\in \mathcal{C}(n,\mathbb{M}) \quad \text{and} \quad x+y=n
$$
and the number of axes of a CoP by
\begin{align}
\nu(n,\mathbb{M}):=\#\lbrace\mathbb{L}_{[x],[y]}\inn \mathcal{C}(n,\mathbb{M})\rbrace.\nonumber
\end{align}
Furthermore, we let 
\begin{align}
\mathbb{N}_n=\left \{m\in \mathbb{N}~:~m\leq n\right \}\nonumber
\end{align}
be the sequence of the first $n$ positive integers. Furthermore, we will denote the \emph{weight} of the point $[x]$ by
\begin{align}
\Vert[x]\Vert:=x\nonumber
\end{align}
 and the weight set of points in CoP $\mathcal{C}(n,\mathbb{M})$ by $||\mathcal{C}(n,\mathbb{M})||$.
\bigskip

\begin{proposition}\label{unique}
Each axis is uniquely determined by points $[x]\in \mathcal{C}(n,\mathbb{M})$. 
\end{proposition}

\begin{proof}
Let $\mathbb{L}_{[x],[y]}$ be an axis of $\mathcal{C}(n,\mathbb{M})$. Suppose that $\mathbb{L}_{[x],[z]}$ is also an axis with $z\neq y$. It implies $n=x+y=x+z$ and therefore $y=z$. This violates property $z\neq y$ of the chosen axes.
\end{proof}

\begin{corollary}\label{partner}
Each point of a CoP $\mathcal{C}(n,\mathbb{M})$ has exactly one axis partner.
\end{corollary}

\begin{proof}
Let $[x]\in \mathcal{C}(n,\mathbb{M})$ be a point without an axis partner. For every point $[y]\neq [x]$, we get
$$
\Vert[x]\Vert+\Vert[y]\Vert\neq n.
$$
This is inconsistent with $[x]\in \mathcal{C}(n,\mathbb{M})$. Due to Proposition \ref{unique} the case of more than one axis partner is impossible.
\end{proof}
\bigskip

\section{The Density of Points on the Circle of Partition}\label{sec:density}

In this section, we introduce the notion of the density of points on CoP $\mathcal{C}(n,\mathbb{M})$ for $\mathbb{M}\subseteq \mathbb{N}$. We use this framework to study the Erd\H{o}s-Tur\'{a}n additive bases conjecture.

\begin{definition}
Let $\mathbb{H}\subset\mathbb{N}$. We denote the density of $\mathbb{H}$ by 
$$
\mathcal{D}\left(\mathbb{H}\right)=\lim_{n\rightarrow\infty}
\frac{\vert\mathbb{H}\cap \mathbb{N}_n\vert}{n}
$$
if the limit exists and it is finite.
\end{definition}

\begin{definition}
Let $\mathbb{H},\mathbb{M}\subset \mathbb{N}$ and $\mathcal{C}(n,\mathbb{M})$ be CoP with $n\in \mathbb{N}$. We denote the density of points $[x]\in \mathcal{C}(n,\mathbb{M})$ such that $x\in \mathbb{H}$ by
\begin{align}
\mathcal{D}\left(\mathbb{H}_{\mathcal{C}(\infty,\mathbb{M})}\right)=\lim \limits_{n\longrightarrow \infty}\frac{\# \lbrace\mathbb{L}_{[x],[y]} \inn \mathcal{C}(n,\mathbb{M})|~ \{x,y\} \cap \mathbb{H}\neq \emptyset \rbrace}{ \nu(n,\mathbb{M})}\nonumber
\end{align}
if the limit exists and it is finite.
\end{definition}
\bigskip

\begin{proposition}\label{propertydensity}
Let $\mathbb{H}\subset \mathbb{M}$ with $\mathbb{M}\subseteq \mathbb{N}$. The following properties hold:
\begin{enumerate}
    \item [(i)] $\mathcal{D}(\mathbb{M}_{\mathcal{C}(\infty,\mathbb{M})})=1$ and $\mathcal{D}(\mathbb{H}_{\mathcal{C}(\infty,\mathbb{M})})\leq 1$.
    \bigskip
    
    \item [(ii)] $1-\lim \limits_{n\longrightarrow \infty}\dfrac{\nu(n,\mathbb{M}\setminus\mathbb{H})}{\nu(n,\mathbb{M})}=\mathcal{D}(\mathbb{H}_{\mathcal{C}(\infty,\mathbb{M})})$.
    \bigskip
    
    \item [(iii)] If $|\mathbb{H}|<\infty$ then $\mathcal{D}(\mathbb{H}_{\mathcal{C}(\infty,\mathbb{M})})=0$.
\end{enumerate}
\end{proposition}

\begin{proof}
It is easy to see that \textbf{Property} $(i)$ and \textbf{Property} $(iii)$ are easy consequences of the definition of the density of points on the CoP $\mathcal{C}(n,\mathbb{M})$. We establish \textbf{Property} $(ii)$, which is the less obvious case. We observe, by the uniqueness of the axes of CoPs, that we can write
\begin{align*}
    1&=\lim \limits_{n\longrightarrow \infty}\frac{\nu(n,\mathbb{M})}{\nu(n,\mathbb{M})}\\
    &=\lim \limits_{n\longrightarrow \infty}\frac{\# \lbrace\mathbb{L}_{[x],[y]}\inn \mathcal{C}(n,\mathbb{M})~|~ x\in \mathbb{H}~,y\in \mathbb{M}\setminus \mathbb{H}\rbrace}{\nu(n,\mathbb{M})}\\
    &+\lim \limits_{n\longrightarrow \infty}\frac{\nu(n,\mathbb{H})}{\nu(n,\mathbb{M})}
    +\lim \limits_{n\longrightarrow \infty}\frac{\nu(n,\mathbb{M}\setminus\mathbb{H})}{\nu(n,\mathbb{M})}\\
    &=\mathcal{D}(\mathbb{H}_{\mathcal{C}(\infty,\mathbb{M})})
    +\lim \limits_{n\longrightarrow \infty}\frac{\nu(n,\mathbb{M}\setminus\mathbb{H})}{\nu(n,\mathbb{M})}
\end{align*}
and $(ii)$ follows immediately.
\end{proof}

\begin{proposition}\label{inequality}
Let $\mathcal{C}(n)$ with $n\in \mathbb{N}$ be a CoP and $\mathbb{H}\subset \mathbb{N}$. The following inequality holds: 
\begin{align}
\mathcal{D}(\mathbb{H})=\lim \limits_{n\longrightarrow \infty}\frac{\left \lfloor \frac{|\mathbb{H}\cap \mathbb{N}_n|}{2}\right \rfloor}{\left \lfloor \frac{n-1}{2}\right \rfloor}\leq \mathcal{D}(\mathbb{H}_{\mathcal{C}(\infty)})\leq \lim \limits_{n\longrightarrow \infty}\frac{|\mathbb{H}\cap \mathbb{N}_n|}{\left \lfloor \frac{n-1}{2}\right \rfloor}=2\mathcal{D}(\mathbb{H}).\nonumber
\end{align}
\end{proposition}

\begin{proof}
The upper bound is obtained from a configuration in which there are no two points $[x],[y]\in \mathcal{C}(n)$ such that $x,y\in \mathbb{H}$ lie on the same axis of the CoP. That is, by the uniqueness of the axes of CoPs with $\nu(n,\mathbb{H})=0$, we can write
\begin{align}
   \# \left \{\mathbb{L}_{[x],[y]}\in \mathcal{C}(n)|~\{x,y\}\cap \mathbb{H}\neq \emptyset \right \}&=\nu(n,\mathbb{H})+\# \left\{\mathbb{L}_{[x],[y]}\in \mathcal{C}(n)|~x\in \mathbb{H},~y\in \mathbb{N}\setminus \mathbb{H}\right \} \nonumber \\&=\# \left\{\mathbb{L}_{[x],[y]}\in \mathcal{C}(n)|~x\in \mathbb{H},~y\in \mathbb{N}\setminus \mathbb{H}\right \} \nonumber \\&=|\mathbb{H}\cap \mathbb{N}_n|.\nonumber
\end{align}
However, the lower bound follows from a configuration in which every two points $[x],[y]\in \mathcal{C}(n)$ with $x,y\in \mathbb{H}$ are joined by an axis of the CoP. That is, by the uniqueness of the axis of CoPs with 
$$
\# \left \{\mathbb{L}_{[x],[y]}\in \mathcal{C}(n)~|~x\in \mathbb{H},~y\in \mathbb{N}\setminus \mathbb{H}\right\}=0
$$ 
we can write 
\begin{align}
    \# \left \{\mathbb{L}_{[x],[y]}\in \mathcal{C}(n)|~\{x,y\}\cap \mathbb{H}\neq \emptyset \right \}&=\nu(n,\mathbb{H})\nonumber \\&=\left \lfloor \frac{|\mathbb{H}\cap \mathbb{N}_n|}{2}\right \rfloor.\nonumber
\end{align}
\end{proof}

\begin{remark}
In the following, we will prove some weaker version of the conjecture by imposing some suitable conditions. The result is encapsulated in the following theorem.
\end{remark}

\begin{theorem}
Let $\mathbb{B}\subset \mathbb{N}$ with 
\begin{align}
\lim \limits_{n\longrightarrow \infty}\frac{|\mathbb{B}\cap \mathbb{N}_n|}{n}>0\nonumber 
\end{align}
such that
\begin{align}
    \# \left \{\mathbb{L}_{[x],[y]}\inn \mathcal{C}(n)|~x\in \mathbb{N}\setminus \mathbb{B},~y\in \mathbb{B}\right\}\leq \nu(n,\mathbb{B}).\nonumber
\end{align}
If $\mathcal{G}_{\mathbb{B}}(n)=\nu(n,\mathbb{B})>0$ for all sufficiently large values of $n$, then 
\begin{align}
    \limsup \limits_{n\longrightarrow \infty} \mathcal{G}_{\mathbb{B}}(n)=\infty.\nonumber
\end{align}
\end{theorem}

\begin{proof}
Suppose $\mathbb{B}\subset \mathbb{N}$ and let $\mathcal{G}_{\mathbb{B}}(n)>0$ for all sufficiently large values of $n$. Suppose that
\begin{align}
  \limsup \limits_{n\longrightarrow \infty} \mathcal{G}_{\mathbb{B}}(n)<\infty.\nonumber  
\end{align}
Consider CoPs $\mathcal{C}(n,\mathbb{B})$. By the uniqueness of axes of CoPs, we can  compute the density of points $[x]\in \mathcal{C}(n)$ with $||[x]||\in \mathbb{B}$ in the following way:
\begin{align*}
\mathcal{D}(\mathbb{B}_{\mathcal{C}(\infty)})
&=\lim \limits_{n\longrightarrow \infty}\frac{\# \lbrace\mathbb{L}_{[x],[y]}\inn \mathcal{C}(n)|\lbrace x,y\rbrace\cap \mathbb{B}\neq \emptyset\rbrace}{\left \lfloor \frac{n-1}{2}\right \rfloor}\\
&=\lim \limits_{n\longrightarrow \infty}\frac{\# \lbrace\mathbb{L}_{[x],[y]}\inn \mathcal{C}(n)|~x\in \mathbb{N}\setminus \mathbb{B},~y\in \mathbb{B}\rbrace}{\left \lfloor \frac{n-1}{2}\right \rfloor}
+\lim \limits_{n\longrightarrow \infty}\frac{\nu(n,\mathbb{B})}{\left \lfloor \frac{n-1}{2}\right \rfloor}\\
&\leq 2\lim \limits_{n\longrightarrow \infty}\frac{\nu(n,\mathbb{B})}{\left \lfloor \frac{n-1}{2}\right \rfloor}\\
&=0
\end{align*}
using the earlier assumption. Applying Proposition \ref{inequality}, we find that
\begin{align}
    \lim \limits_{n\longrightarrow \infty}\frac{\left \lfloor \frac{|\mathbb{B}\cap \mathbb{N}_n|}{2}\right \rfloor}{\left \lfloor \frac{n-1}{2}\right \rfloor}=0.\nonumber
\end{align}
We deduce $\mathcal{D}(\mathbb{B})=0$, contradicting the requirement of the statement.
\end{proof}
\bigskip

\section{Ascending, Descending and Stationary Circles of Partition}\label{sec:ascdesc}

In this section, we introduce the notion of \emph{ascending},  \emph{descending}, and \emph{stationary} CoPs between generators. We exploit this framework to improve the result concerning the Erd\H{o}s-Tur\'{a}n additive bases conjecture in section 2.

\begin{definition}
Let $\mathbb{M}\subset\mathbb{N}$ with $\mathcal{C}(n,\mathbb{M})$ be a CoP. We say that the CoP $\mathcal{C}(n,\mathbb{M})$ is \emph{ascending} from $n$ to the \emph{spot} $m$ if for $n<m$ holds
\begin{align*}
\nu(n,\mathbb{M})<\nu(m,\mathbb{M}).
\end{align*}
Similarly, we say that it is \emph{descending} from $n$ to the \emph{spot} $m$ if for $n<m$, then 
\begin{align*}
\nu(n,\mathbb{M})>\nu(m,\mathbb{M}).
\end{align*}
We say that it is \emph{globally} ascending (resp. descending) if at  $\forall m \in \mathbb{N}$ it is ascending (resp. descending). We say that CoP $\mathcal{C}(n,\mathbb{M})$ is \emph{stationary} from $n$ to the \emph{spot} $m$ if for $n<m$ then 
\begin{align*}
   \nu(n,\mathbb{M})=\nu(m,\mathbb{M}). 
\end{align*}
Similarly, we say that it is \emph{globally stationary} if it is stationary at all spots $m\in \mathbb{N}$. If CoP $\mathcal{C}(n,\mathbb{M})$ is neither globally ascending, descending, nor stationary, then we say that it is \emph{globally oscillatory}.
\end{definition}

\begin{theorem}\label{ascendingspots}
Let $\mathbb{H}\subset \mathbb{N}$ and $\mathcal{C}(n,\mathbb{H})$ be a CoP. If 
\begin{align}
    \lim \limits_{n\longrightarrow \infty}\frac{|\mathbb{H}\cap \mathbb{N}_n|}{n}>0\nonumber
\end{align}
with 
\begin{align}
    \lim \limits_{n\longrightarrow \infty}\frac{|(\mathbb{N}\setminus \mathbb{H})\cap \mathbb{N}_n|}{n}<\frac{1}{2}\lim \limits_{n\longrightarrow \infty}\frac{|\mathbb{H}\cap \mathbb{N}_n|}{n}\nonumber
\end{align}
then $\mathcal{C}(n,\mathbb{H})$ is ascending at infinitely many spots.
\end{theorem}

\begin{proof}
Let $\mathcal{C}(n,\mathbb{H})$ be a CoP and assume to the contrary that there are finitely many \emph{spots} at which it is ascending. Let us name and arrange the spots as follows $m_1<m_2<\cdots<m_k$. It implies that 
\begin{align}
\nu(n,\mathbb{H})\geq \nu(m_{k+1},\mathbb{H})\geq \cdots \nu(m_{k+i},\mathbb{H})\geq \cdots \nonumber
\end{align}
for all $i\geq 1$. The upshot is that 
\begin{align}
    \lim \limits_{n\longrightarrow \infty}\frac{\nu(n,\mathbb{H})}{\left \lfloor \frac{n-1}{2}\right \rfloor}=0.\nonumber
\end{align}
By the uniqueness of axes of CoPs, we can compute the density of points with weight in $\mathbb{H}$ on the CoP $\mathcal{C}(n)$ as follows:
\begin{align*}
    \mathcal{D}(\mathbb{H}_{\mathcal{C}(\infty)})&=\lim \limits_{n\longrightarrow \infty}\frac{\#\left \{\mathbb{L}_{[x],[y]}\inn \mathcal{C}(n)|~\{x,y\}\cap \mathbb{H}\neq \emptyset\right \}}{\left \lfloor \frac{n-1}{2}\right \rfloor}\\
    &=\lim \limits_{n\longrightarrow \infty}\frac{\#\left \{\mathbb{L}_{[x],[y]}\inn \mathcal{C}(n)|~x\in  \mathbb{H},~y\in \mathbb{N}\setminus \mathbb{H}\right\}}{\left \lfloor \frac{n-1}{2}\right \rfloor}
    +\lim \limits_{n\longrightarrow \infty}\frac{\nu(n,\mathbb{H})}{\left \lfloor \frac{n-1}{2}\right \rfloor}\\
    &=\lim \limits_{n\longrightarrow \infty}\frac{\#\left \{\mathbb{L}_{[x],[y]}\inn \mathcal{C}(n)|~x\in  \mathbb{H},~y\in \mathbb{N}\setminus \mathbb{H}\right\}}{\left \lfloor \frac{n-1}{2}\right\rfloor}\\
    &\leq \lim \limits_{n\longrightarrow \infty}\frac{|(\mathbb{N}\setminus \mathbb{H})\cap \mathbb{N}_n|}{\left \lfloor \frac{n-1}{2}\right \rfloor}\\
    &\leq 2\lim \limits_{n\longrightarrow \infty}\frac{|(\mathbb{N}\setminus \mathbb{H})\cap \mathbb{N}_n|}{n}.
\end{align*}
Applying Proposition \ref{inequality}, we deduce
\begin{align}
    \lim \limits_{n\longrightarrow \infty}\frac{|\mathbb{H}\cap \mathbb{N}_n|}{n}\leq 2\lim \limits_{n\longrightarrow \infty}\frac{|(\mathbb{N}\setminus \mathbb{H})\cap \mathbb{N}_n|}{n}.\nonumber
\end{align}
This violates the statement requirements.
\end{proof}

\begin{remark}
We can easily deduce from this result another weak variant of Erd\H{o}s-Tur\'{a}n conjecture. Roughly speaking, it purports that very dense sequences sufficiently qualify as an additive base.
\end{remark}

\begin{corollary}
Let $\mathbb{H}\subset \mathbb{N}$ with $\mathcal{D}(\mathbb{H})>0$ such that $\mathcal{D}(\mathbb{N}\setminus \mathbb{H})<\frac{1}{2}\mathcal{D}(\mathbb{H})$. If 
\begin{align}
    r_{\mathbb{H}}(n):=\#\left \{(a,b)\in \mathbb{H}^2|~a+b=n\right\}\nonumber
\end{align}
then $\lim \limits_{n\longrightarrow \infty}r_{\mathbb{H}}(n)=\infty$.
\end{corollary}

\section{The $l^{th}$ Fold Energy of Circles of Partition}\label{sec:energy}

In this section, we introduce and study the notion of the $l^{th}$ fold energy of CoPs and exploit some applications in this context. This notion tends to be more effective and extends very much to sequences that do not necessarily have a positive density.

\begin{definition}
Let $\mathbb{M}\subset \mathbb{N}$ and $\mathcal{C}(n,\mathbb{M})$ be a CoP. We denote the $l^{th}$-fold energy of the CoP $\mathcal{C}(n,\mathbb{M})$ by the measure
\begin{align}
    \mathcal{E}(l,\mathbb{M}):=\sum \limits_{n=3}^{\infty} \frac{\nu(n^l,\mathbb{M})}{\left \lfloor \frac{n^l-1}{2}\right \rfloor}\nonumber
\end{align}
for a fixed $l\in \mathbb{N}$.
\end{definition}
\bigskip

We note that the $l^{th}$-fold energy of a typical CoP $\mathcal{C}(n,\mathbb{M})$ could be infinite or finite. In the latter case, it should have a finite value. To that effect, we state the following proposition.

\begin{proposition}\label{boundedenergy}
Let $\mathbb{J}^l\subset \mathbb{N}$ be the set of all $l^{th}$ powers. We have $\mathcal{E}(l,\mathbb{J}^l)<\infty$ for all $l\geq 3$ and $\mathcal{E}(2,\mathbb{J}^2)=\infty$. 
\end{proposition}

\begin{proof}
Let $l\geq 3$ be fixed and consider the CoP $\mathcal{C}(n^l,\mathbb{J}^l)$, where $\mathbb{J}^l\subset \mathbb{N}$ is the set of all $l^{th}$ powers. We deduce from the configuration of CoPs the following inequality
\begin{align}
    \mathcal{E}(l,\mathbb{J}^l)&=\sum \limits_{n=3}^{\infty}\frac{\nu(n^l,\mathbb{J}^l)}{\left \lfloor \frac{n^l-1}{2}\right \rfloor}\nonumber\\
    &\leq \sum \limits_{n=3}^{\infty}\frac{\frac{n}{2}}{\left \lfloor \frac{n^l-1}{2}\right \rfloor}\nonumber\\
    & \ll \sum \limits_{n=3}^{\infty}\frac{1}{n^{l-1}}<\infty \nonumber
\end{align}
for all $l\geq 3$.
\end{proof}

\begin{proposition}\label{spotenergy}
Let $\mathbb{M}\subset \mathbb{N}$ and $\mathcal{C}(n,\mathbb{M})$ be a CoP. If $\mathcal{E}(l,\mathbb{M})=\infty$ for $l\geq 2$, then $\mathcal{C}(n^l,\mathbb{M})$ ascends at infinitely many spots.
\end{proposition}

\begin{proof}
Let $\mathcal{E}(l,\mathbb{M})=\infty$ and assume that CoP $\mathcal{C}(n^l,\mathbb{M})$ is ascending at finitely many \emph{spots}. We deduce 
\begin{align*}
\lim \limits_{n\longrightarrow \infty}\nu(n^l,\mathbb{M})<\infty.
\end{align*}
This implies $\mathcal{E}(l,\mathbb{M})<\infty$, which violates the requirement of the statement.
\end{proof}

\section{Remarks}\label{sec:main}

The Erd\H{o}s-Tur\'{a}n additive bases conjecture is the assertion:

\begin{conjecture}[Erd\H{o}s-Tur\'{a}n]
Let $\mathbb{B}\subset \mathbb{N}$ and 
\begin{align}
  r_{\mathbb{B}}(n):=\# \left \{(a,b)\in \mathbb{B}^2|~ a+b=n\right \}.\nonumber
\end{align}
If $r_{\mathbb{B}}(n)>0$ for all sufficiently large values of $n$, then \begin{align}
    \limsup \limits_{n\longrightarrow \infty}~r_{\mathbb{B}}(n)=\infty.\nonumber
\end{align}
\end{conjecture}
\bigskip

Using the framework of the \emph{circle of partitions}, we formulate a stronger version of the Erd\H{o}s-Tur\'{a}n additive bases conjecture.

\begin{conjecture}
Let $\mathbb{B}\subset \mathbb{N}$ with $\# \left \{n\leq x|~n\in \mathbb{B}\right\}\sim x^{1-\epsilon}$ for any $0<\epsilon\leq \frac{1}{2}$ and set
\begin{align*}
   \mathcal{G}_{\mathbb{B}}(n):=\nu(n,\mathbb{B})
\end{align*}
If $\mathcal{G}_{\mathbb{B}}(n)>0$ for all  sufficiently large values of $n$, then
\begin{align}
    \limsup \limits_{n\longrightarrow \infty}\mathcal{G}_{\mathbb{B}}(n)=\infty.\nonumber
\end{align}
\end{conjecture}

\footnote{
\par
.}%

\bibliographystyle{amsplain}

\end{document}